\documentclass[12pt]{amsart}

\usepackage[english]{babel}

\usepackage[
pdftex,hyperfootnotes]{hyperref}
\hypersetup{
	colorlinks=true,
	linkcolor=NavyBlue, 
	urlcolor=RoyalPurple,
	citecolor=OliveGreen,
	pdftitle={Bidiagonal factorization of tetradiagonal  matrices and Darboux transformations },
	bookmarks=true,
	pdfpagemode=FullScreen,
}
\usepackage[usenames,dvipsnames,svgnames,table,x11names]{xcolor}
\usepackage{pgfplots}
\usepackage{tikz}
\usepackage{tikz-3dplot}
\usetikzlibrary{automata,quotes, chains,matrix,calc,shadows,shapes.callouts,shapes.geometric,shapes.misc,positioning,patterns,decorations.shapes,
	decorations.pathmorphing,decorations.markings,decorations.fractals,decorations.pathreplacing,shadings,fadings,arrows.meta,bending}

\usepackage{nicematrix}
\usepackage[utf8]{inputenc}
\usepackage{polski}
\usepackage{comment}
\usepackage[textwidth=17.5cm,textheight=22.275cm,
height=
24.275cm,
width=18cm]{geometry}

\usepackage{rotating,multirow,pdflscape}
\usepackage{amssymb,latexsym,amsmath,amsthm,bm}
\usepackage{mathrsfs}
\usepackage{mathtools,arydshln,mathdots}
\usepackage{bigints}
\makeatletter
\renewcommand*\env@matrix[1][*\c@MaxMatrixCols c]{%
	\hskip -\arraycolsep
	\let\@ifnextchar\new@ifnextchar
	\array{#1}}
\makeatother

\mathtoolsset{centercolon}
\usepackage[table]{xcolor}
\usepackage[toc,page]{appendix}

\usepackage{tocvsec2}

\usepackage{Baskervaldx}
\usepackage[]{newtxmath }
\newtheorem{coro}{Corollary}

\newtheorem{defi}{Definition}
\newtheorem{teo}{Theorem}

\newtheorem{pro}{Proposition}
\newtheorem{lemma}{Lemma}

\newtheorem{rem}{Remark}

\renewcommand{\d}{\operatorname{d}}

\newcommand{\N}{\mathbb{N}}
\newcommand{\R}{\mathbb{R}}
\newcommand{\Z}{\mathbb{Z}}

\allowdisplaybreaks[1]

\makeatletter
\DeclareRobustCommand{\gaussk}{\DOTSB\gaussk@\slimits@}
\newcommand{\gaussk@}{\mathop{\vphantom{\sum}\mathpalette\bigcal@{K}}}

\newcommand{\bigcal@}[2]{%
	\vcenter{\m@th
		\sbox\z@{$#1\sum$}%
		\dimen@=\dimexpr\ht\z@+\dp\z@
		\hbox{\resizebox{!}{0.8\dimen@}{$\mathcal{K}$}}%
	}%
}
\newcommand{\cfracplus}{\mathbin{\cfracplus@}}
\newcommand{\cfracplus@}{%
	\sbox\z@{$\dfrac{1}{1}$}%
	\sbox\tw@{$+$}%
	\raisebox{\dimexpr\dp\tw@-\dp\z@\relax}{$+$}%
}
\newcommand{\cfracdots}{\mathord{\cfracdots@}}
\newcommand{\cfracdots@}{%
	\sbox\z@{$\dfrac{1}{1}$}%
	\sbox\tw@{$+$}%
	\raisebox{\dimexpr\dp\tw@-\dp\z@\relax}{$\cdots$}%
}
\makeatother
\begin{document}

\title[]
{Bidiagonal factorization of tetradiagonal  matrices and Darboux transformations }

\author[A Branquinho]{Amílcar Branquinho$^{1,\flat}$}
\address{$^1$Departamento de Matemática,
	Universidade de Coimbra, 3001-454 Coimbra, Portugal}
\email{ajplb@mat.uc.pt}

\author[A Foulquié]{Ana Foulquié-Moreno$^{2,\natural}$}
\address{$^2$Departamento de Matemática, Universidade de Aveiro, 3810-193 Aveiro, Portugal}
\email{foulquie@ua.pt}

\author[M Mañas]{Manuel Mañas$^{3,\clubsuit}$}
\address{$^3$Departamento de Física Teórica, Universidad Complutense de Madrid, Plaza Ciencias 1, 28040-Madrid, Spain \&
	Instituto de Ciencias Matematicas (ICMAT), Campus de Cantoblanco UAM, 28049-Madrid, Spain}
\email{manuel.manas@ucm.es}

\thanks{$^\flat$Acknowledges Centro de Matemática da Universidade de Coimbra (CMUC) -- UID/MAT/00324/2019, funded by the Portuguese Government through FCT/MEC and co-funded by the European Regional Development Fund through the Partnership Agreement PT2020}

\thanks{$^\natural$Acknowledges CIDMA Center for Research and Development in Mathematics and Applications (University of Aveiro) and the Portuguese Foundation for Science and Technology (FCT) within project UIDB/MAT/UID/04106/2020 and UIDP/MAT/04106/2020}

\thanks{$^\clubsuit$Thanks financial support from the Spanish ``Agencia Estatal de Investigación'' research project [PGC2018-096504-B-C33], \emph{Ortogonalidad y Aproximación: Teoría y Aplicaciones en Física Matemática} and [PID2021- 122154NB-I00], \emph{Ortogonalidad y Aproximación con Aplicaciones en Machine Learning y Teoría de la Probabilidad}.}



\begin{abstract}
		Recently a spectral Favard theorem for  bounded banded lower Hessenberg matrices that admit a  positive bidiagonal factorization was presented. 
		These type of matrices are oscillatory.
In this paper the Lima--Loureiro hypergeometric multiple orthogonal polynomials   and the Jacobi--Piñeiro multiple orthogonal polynomials are discussed at the light of this bidiagonal factorization for tetradiagonal matrices. The Darboux transformations of  tetradiagonal Hessenberg matrices is studied  
and Christoffel formulas for the elements of the bidiagonal factorization are given,  i.e.,  the bidiagonal factorization is given in terms of the recursion polynomials  evaluated at the origin.
	\end{abstract}
	
	\maketitle
	
	
	\section{Introduction}

In this paper we will analyze some aspects  for  the tetradiagonal Hessenberg matrix of  the form
\begin{align}\label{eq:monic_Hessenberg}
	T&=\left[\begin{NiceMatrix}[columns-width = auto]
		c_0 & 1&0&\Cdots&\\[-3pt]
		b_1&c_1 & 1&\Ddots&&\\
		a_2&b_2&c_2&1&&\\
		0&a_3&b_3&c_3&1&\\
		\Vdots&\Ddots&\Ddots&\Ddots&\Ddots&\Ddots\\&&&&
	\end{NiceMatrix}\right], 
\end{align}
where we assume that $a_n>0$, and its  bidiagonal factorization
	\begin{align}\label{eq:bidiagonal}
	T=  L_{1}  L_{2}  U,
\end{align}
with bidiagonal matrices  given by 
\begin{align}\label{eq:LU_J_oscillatory_factors}
	L_1&=
	\left[\begin{NiceMatrix}[columns-width = auto]
		1 &0&0&\Cdots&\\
		\alpha_{2} & 1 &0&\Cdots\\
		0& \alpha_{5} & 1& \Ddots\\
		\Vdots& \Ddots& \Ddots & \Ddots
	\end{NiceMatrix}\right],& L_2&=
	\left[\begin{NiceMatrix}[columns-width = auto]
		1 &0&0&\Cdots&\\
		\alpha_{3} & 1 &0&\Cdots\\
		0& \alpha_{6} & 1& \Ddots\\
		\Vdots& \Ddots& \Ddots & \Ddots
	\end{NiceMatrix}\right] ,&U& =
	\left[\begin{NiceMatrix}[columns-width = auto]
		\alpha_1 & 1 &0&\Cdots&\\
		0& \alpha_4 & 1 &\Ddots&\\
		0&0&\alpha_7&\Ddots&\\
		\Vdots& \Ddots& \Ddots &\Ddots &
	\end{NiceMatrix}\right].
\end{align}
If the requirement 
\begin{align*}
	\alpha_j&>0, & j\in\N,
	\end{align*}
is fulfilled, we say that we have a  \emph{positive bidiagonal factorization} (PBF).
 In \cite{previo}, this factorization was shown to be the key for a Favard theorem for bounded banded Hessenberg semi-infinite matrices and the existence of  positive measures such that the recursion polynomials are multiple orthogonal polynomials and the Hessenberg matrix is the recursion matrix from this set of multiple orthogonal polynomials. We also gave a multiple Gauss quadrature together with explicit degrees of precision. Then, in \cite{previo2} we studied for the tetradiagonal case when such PBF, in terms of continued fractions, for oscillatory matrices exists.
 Oscillatory tetradiagonal Toeplitz matrices were shown to  admits a PBF. Moreover, it was proven that oscillatory banded Hessenberg matrices are organized in rays, with the origin of the ray not having the positive bidiagonal factorization and all the interior points of the ray having such positive bidiagonal factorization. 
 
In the next  section of this paper we  succinctly discuss two cases that appear in the literature, the Jacobi--Piñeiro \cite{pineiro,Ismail} and the hypergeometric \cite{Lima-Loureiro,nuestro2} families.
In the final section, the Darboux and Christoffel transformations \cite{bfm} are connected with the PBF factorization,  the coefficients  $\alpha$ in the PBF are reconstructed in terms of the values of the type II and I polynomials at $0$, see Theorem \ref{teo:alphasBA} and the Darboux transformations are discussed at the  light of  the spectral Christoffel perturbations \cite{afm}, see Theorem \ref{theorem:3}. 
 
 \subsection{Preliminary material}
 
 For multiple orthogonal polynomials see \cite{nikishin_sorokin,Ismail} and \cite{afm}.
 
 Let us denote by $ T^{[N]}=T[\{0,1,\dots,N\}]\in\R^{(N+1)\times(N+1)}$ the $(N+1)$-th leading principal  submatrix  of the banded Hessenberg  matrix $T$: 
 \begin{align}\label{eq:Hessenberg_truncation}
 	T^{[N]}\coloneq \begin{bNiceMatrix}[columns-width = auto]
 		c_0 & 1&0&\Cdots&&&0\\[-3pt]
 		b_1&c_1&1&\Ddots&&&\Vdots\\
 		a_2&b_2&c_2&1&&&\\
 		0&a_3&b_3&c_3&1&&\\
 		\Vdots&\Ddots&\Ddots&\Ddots&\Ddots& \Ddots&0\\
 		&&&a_{N-1}&b_{N-1}&c_{N-1}&1\\
 		0&\Cdots&&0&a_{N}&b_{N}&c_N
 	\end{bNiceMatrix}.
 \end{align}

\begin{defi}[Recursion polynomials of type II]\label{def:typeII}
The type II recursion vector of polynomials
\begin{align*}
B(x)&=
\begin{bNiceMatrix}[columns-width = auto]
B_0(x)\\
B_1(x)\\
\Vdots
\end{bNiceMatrix}, & \deg B_n=n,
\end{align*}
is determined by the following eigenvalue equation
\begin{align}\label{eq:recurrence_J}
TB(x)=x B(x).
\end{align}
Uniqueness is ensured by taking as initial condition $B_0=1$.
We call the components $B_n$  type II recursion polynomials. One obtains that
 $B_1=x- c_0$, $B_2=(x-c_0)(x-c_1)-b_1$, and  higher degree recursion polynomials are constructed  by means of the 4-term recurrence relation
\begin{align}\label{eq:recurrence_B}
B_{n+1}&=(x-c_n)B_n-b_nB_{n-1}-a_n B_{n-2}, & n&\in\{2,3,\dots\}.
\end{align}
\end{defi}

\begin{defi}[Recursion polynomials of type I]\label{def:typeI}
Dual to the polynomial vector $B(x)$ we consider the two following    polynomial dual vectors
\begin{align*}
	A^{(1)}(x)&=\left[\begin{NiceMatrix}
			A^{(1)}_0(x) & 	A^{(1)}_1(x)&\Cdots
		\end{NiceMatrix}\right], & 
	A^{(2)}(x)&=\left[\begin{NiceMatrix}
			A^{(2)}_0(x) & 	A^{(2)}_1(x)&\Cdots
		\end{NiceMatrix}\right], 
\end{align*}
that are left eigenvectors of the semi-infinite matrix  $J$, i.e.,
\begin{align*}
	A^{(1)}(x) T&=x 	A^{(1)}(x), & A^{(2)}(x) T&=xA^{(2)}(x).
\end{align*} 
The initial conditions, that determine these polynomials uniquely,  are taken as
\begin{align*}
	A^{(1)}_0&=1, & A^{(1)}_1&=\nu,&
	A^{(2)}_0&=0, & A^{(2)}_1&=1,
\end{align*}
with $\nu\neq 0$ being an arbitrary constant. Then, from the first the relation
\begin{align*}
	c_0A^{(a)}_0+b_{1}A^{(a)}_{1}+a_2A^{(a)}_{2}&=xA^{(a)}_0, & a&\in\{1,2\},
\end{align*}
we get
$A^{(1)}_2=\frac{x}{a_2}-\frac{c_0+b_1\nu}{a_2}$ and
$A^{(2)}_2=-\frac{b_1}{a_2}$.
The other polynomials in these sequences are determined by the following four term recursion relation
\begin{align}\label{eq:recursion_dual_A}
A^{(a)}_n	a_n&=-A^{(a)}_{n-1}b_{n-1}+A^{(a)}_{n-2}(x-c_{n-2})-A^{(a)}_{n-3}, & n&\in\{3,4,\dots\}, & a&\in\{1,2\}.
\end{align}
\end{defi}
For example, one finds
\begin{align*}
A^{(1)}_3	a_3&=-A^{(1)}_2b_2+A^{(1)}_1(x-c_1)-A^{(1)}_0=
	-b_2\Big(\frac{x}{a_2}-\frac{c_0+b_1\nu}{a_2}\Big)+\nu (x-c_1)-1,\\
A^{(2)}_3	a_3 &=-A^{(2)}_2b_2+A^{(2)}_1(x-c_1)-A^{(2)}_0=
	b_2\frac{b_1}{a_2}+x-c_1. 
\end{align*}


Second kind polynomials are also relevant in the theory of multiple orthogonality.
\begin{defi}[Recursion polynomials of type II of the second kind]
Let us consider the recursion relation  \eqref{eq:recurrence_B} in the form
\begin{align}\label{eq:recurrence_B_2}
a_nB_{n-2}+b_nB_{n-1}+c_nB_n+B_{n+1}=xB_n,
\end{align}
 set $b_0=a_0=a_1=-1$ and   $n\in\N_0$. The values, \emph{initial conditions}, for $B_{-2}, B_{-1}, B_0$ are required to get the values $B_n$ for  $n\in\N$. The polynomials of type II correspond to the choice
\begin{align}
B_{-2}&=0, &B_{-1}&=0, &B_0&=1.
\end{align}
Two sequences of polynomials of type II of the second kind $\big\{B_n^{(1)}\big\}_{n=0}^\infty$ and $\big\{B_n^{(2)}\big\}_{n=0}^\infty$  are defined by the following initial conditions
\begin{align}
B^{(1)}_{-2}&=1, &B^{(1)}_{-1}&=0, &B^{(1)}_0&=0,\\
B^{(2)}_{-2}&=-1-\nu, &B^{(2)}_{-1}&=1, &B^{(2)}_0&=0.
\end{align}
\end{defi}
\begin{pro}[Determinantal expressions]\label{pro:determintal_second_kind}
	\begin{enumerate}
		\item 	For the recursion polynomials we have the determinantal expressions
		\begin{align}\label{eq:det_B}
			B_{N+1}=\det\big(x I_{N+1}- T^{[N]}\big)=\begin{vNiceMatrix}[columns-width = 10pt]
				x-c_0 & -1&0&\Cdots&&&0\\[-3pt]
				-b_1&x-c_1&-1&\Ddots&&&\Vdots\\
				-a_2&-b_2&x-c_2&-1&&&\\
				0&-a_3&-b_3&x-c_3&-1&&\\
				\Vdots&\Ddots&\Ddots&\Ddots&\Ddots& \Ddots&0\\
				&&&-a_{N-1}&-b_{N-1}&x-c_{N-1}&-1\\
				0&\Cdots&&0&-a_{N}&-b_{N}&x-c_N
			\end{vNiceMatrix}.
		\end{align}
		Hence, they are the characteristic polynomials of the leading principal submatrices $T^{[N]}$.
		\item For the  recursion polynomials of type II of the second kind,  $B^{(1)}_{N+1} $ and $B^{(2)}_{N+1} $,  we  have the following adjugate and determinantal expressions 
		\begin{align*}
			B^{(1)}_{N+1} &= e_1^\top \operatorname{adj}\big(x I_{N+1}-T^{[N]}\big)e_1=\begin{vNiceMatrix}[columns-width = 10pt]
				x-c_1 & -1&0&\Cdots&&&0\\[-3pt]
				-b_2&x-c_2&-1&\Ddots&&&\Vdots\\
				-a_3&-b_3&x-c_3&-1&&&\\
				0&-a_4&-b_4&x-c_4&-1&&\\
				\Vdots&\Ddots&\Ddots&\Ddots&\Ddots& \Ddots&0\\
				&&&-a_{N-1}&-b_{N-1}&x-c_{N-1}&-1\\
				0&\Cdots&&0&-a_{N}&-b_{N}&x-c_N
			\end{vNiceMatrix},\\
			B^{(2)}_{N+1} &=  e_1^\top \operatorname{adj}\big(x I_{N+1}-T^{[N]}\big)(e_2-\nu e_1)
			=b^{(1)}_{N+1}-\nu B^{(1)}_{N+1} , \\ b^{(1)}_{N+1}&= e_1^\top \operatorname{adj}\big(x I_{N+1}-T^{[N]}\big)e_2=
			\begin{vNiceMatrix}[columns-width = 10pt]
				x-c_2 & -1&0&\Cdots&&&0\\[-3pt]
				-b_3&x-c_3&-1&\Ddots&&&\Vdots\\
				-a_4&-b_4&x-c_4&-1&&&\\
				0&-a_5&-b_5&x-c_5&-1&&\\
				\Vdots&\Ddots&\Ddots&\Ddots&\Ddots& \Ddots&0\\
				&&&-a_{N-1}&-b_{N-1}&x-c_{N-1}&-1\\
				0&\Cdots&&0&-a_{N}&-b_{N}&x-c_N
			\end{vNiceMatrix}.
		\end{align*}
	\end{enumerate}
\end{pro}

Finite truncations of this matrices having this positive bidiagonal factorization are oscillatory matrices. In fact,  we will be dealing in this paper with totally non negative matrices and oscillatory matrices and, consequently,  we require of some definitions and  properties that we are about to present succinctly.

%

Further truncations are
\begin{align}\label{eq:TNk}
	T^{[N,k]}&\coloneq \begin{bNiceMatrix}[columns-width = 10pt]
		c_k & 1&0&\Cdots&&&0\\[-3pt]
		b_{k+1}&c_{k+1}&1&\Ddots&&&\Vdots\\
		a_{k+2}&b_{k+2}&c_{k+2}&1&&&\\
		0&a_{k+3}&b_{k+3}&c_{k+3}&1&&\\
		\Vdots&\Ddots&\Ddots&\Ddots&\Ddots& \Ddots&0\\
		&&&a_{N-1}&b_{N-1}&c_{N-1}&1\\
		0&\Cdots&&0&a_{N}&b_{N}&c_N
	\end{bNiceMatrix}\in\R^{(N+1-k)\times(N+1-k)}, & k&\in\{0,1,\dots,N\},\\
	T^{[N,N+1]}&\coloneq 1,
\end{align}
notice that $T^{[N]}=T^{[N,0]}$. 
Corresponding characteristic polynomials are:
\begin{align}\label{eq:BNk}
B^{[k]}_{N+1}&\coloneq \begin{vNiceMatrix}[columns-width = 10pt]
			x-c_k & -1&0&\Cdots&&&0\\[-3pt]
			-b_{k+1}&x-c_{k+1}&-1&\Ddots&&&\Vdots\\
			-a_{k+2}&-b_{k+2}&x-c_{k+2}&-1&&&\\
			0&-a_{k+3}&-b_{k+3}&x-c_{k+3}&-1&&\\
			\Vdots&\Ddots&\Ddots&\Ddots&\Ddots& \Ddots&0\\
			&&&-a_{N-1}&-b_{N-1}&x-c_{N-1}&-1\\
			0&\Cdots&&0&-a_{N}&-b_{N}&x-c_N
		\end{vNiceMatrix}, & k&\in\{0,1,\dots,N\}.
\end{align}

Totally nonnegative (TN) matrices are those with all their minors nonnegative \cite{Fallat-Johnson,Gantmacher-Krein}, and the set of nonsingular TN matrices is denoted by InTN. Oscillatory matrices \cite{Gantmacher-Krein} are totally nonnegative, irreducible \cite{Horn-Johnson} and nonsingular. Notice that the set of oscillatory matrices is denoted by IITN (irreducible invertible totally nonnegative) in \cite{Fallat-Johnson}. An oscillatory matrix~$A$ is equivalently defined as a totally nonnegative matrix $A$ such that for some $n$ we have that $A^n$ is totally positive (all minors are positive). From Cauchy--Binet Theorem one can deduce the invariance of these sets of matrices under the usual matrix product. Thus, following \cite[Theorem 1.1.2]{Fallat-Johnson} the product of matrices in InTN is again InTN (similar statements hold for TN or oscillatory matrices).
We have the important result:
\begin{teo}[Gantmacher--Krein Criterion] \cite[Chapter 2, Theorem 10]{Gantmacher-Krein}. \label{teo:Gantmacher--Krein Criterion}
	A~totally non negative matrix $A$ is oscillatory if and only if it is nonsingular and the elements at the first subdiagonal and first superdiagonal are positive.
\end{teo}

The Gauss--Borel factorization of the  matrix $T^{[N]}$ in  \eqref{eq:Hessenberg_truncation} is the following factorization 
\begin{align}\label{eq:Gauss-Borel}
	T^{[N]}=L^{[N]} U^{[N]}
\end{align}
with banded triangular matrices given by 
\begin{align*}
	L^{[N]}&=\left[\begin{NiceMatrix}[columns-width = auto]
		1 &0&\Cdots&&&0\\
		\mathscr m_1& 1 &\Ddots&&&\Vdots\\
		\mathscr l_2 &\mathscr m_2& 1& &&\\
		0&\mathscr l_3&\mathscr m_3&1&&\\
		\Vdots& \Ddots& \Ddots & \Ddots&\Ddots&0\\
		0	&\Cdots&0&\mathscr l_N&\mathscr m_N&1
	\end{NiceMatrix}\right], & U^{[N]}& =
	\left[\begin{NiceMatrix}
		\alpha_1 & 1 &0&\Cdots&0\\
		0& \alpha_4 & \Ddots &\Ddots&\Vdots\\
		\Vdots&\Ddots&\alpha_7&&0\\
		& & \Ddots &\Ddots &1\\
		0 &\Cdots&&0&\alpha_{3N+1}
	\end{NiceMatrix}\right].
\end{align*}

\begin{pro}
	The Gauss--Borel factorization  exists if and only if all  leading  principal minors $\delta^{[N]}$ of $T^{[N]}$ are not zero.
	For $n\in\N$, the following expressions  for the coefficients hold
	\begin{align}\label{eq:gauss_borel_a_b_alpha}
		\mathscr l_{n+1}&=\frac{a_{n+1}\delta^{[n-2]}}{\delta^{[n-1]}},&
		\mathscr m_n&=c_n-\frac{\delta^{[n]}}{\delta^{[n-1]}},&
		\alpha_{3n-2}&=\frac{\delta^{[n-1]}}{\delta^{[n-2]}},
	\end{align}
	where $\delta^{[-1]}=1$ and $a_1=0$,  and we have the following recurrence relation for the determinants
	\begin{align}\label{eq:recurrence_delta}
		\delta^{[n]}=	a_n\delta^{[n-3]}-b_n\delta^{[n-2]}+c_n\delta^{[n-1]},
	\end{align}
	is satisfied.
	
\end{pro}

For oscillatory matrices the Gauss--Borel factorization exits and both triangular factors belong to InTN. See \cite{manas} for a modern account of the role of Gauss--Borel factorization problem in the realm of standard  and non standard orthogonality.


\section{Hypergeometric and Jacobi--Piñeiro examples}
Now we discuss two cases of tetradiagonal Hessenberg matrices that appear as recurrence matrices of two families of  multiple orthogonal polynomials. For each of them we consider bidiagonal factorizations and its positivity.

\subsection{{Hypergeometric multiple orthogonal polynomials}}
In \cite{Lima-Loureiro} a new set of multiple hypergeometric polynomials were introduced by Lima and Loureiro. 
The corresponding recursion matrix $T_{LL}$ was used in \cite{nuestro1} to construct stochastic matrices and associated Markov chains beyond birth and death. For the hypergeometric case the $T_{LL}=L_1L_2U$ bidiagonal factorization is provided in \cite[Equations 107-110]{Lima-Loureiro}, that ensures the regular oscillatory character of the  matrix $T_{LL}$ for this hypergeometric case. Notice the correspondence between Lima--Loureiro's $\lambda_n$  and our $\alpha_n$ is $	\lambda_{3n+2}\to \alpha_{3n+1}$, $\lambda_{3n+1}\to \alpha_{3n}$ and $\lambda_{3n}\to \alpha_{3n-1}$.
These coefficients were gotten in \cite{Lima-Loureiro} from \cite[Theorem 14.5]{Petreolle-Sokal-Zu} as the coefficients of a branched-continued-fraction representation for $_{3}F_{2}$. For more on this see the recent paper \cite{Zu}. This sequence is TP, so that $T_{LL}$ is a regular oscillatory banded Hessenberg matrix.

\subsection{{Jacobi--Piñeiro multiple orthogonal polynomials}}\label{Section:JP}
Jacobi--Piñeiro multiple orthogonal polynomials, associated with weights $w_1=x^{\alpha}(1-x)^\gamma, w_2=x^{\beta}(1-x)^\gamma$ with support on $[0,1]$ $\alpha,\beta,\gamma>-1$, $\alpha-\beta\not\in\Z$, is a well study case.
This system is an AT system and the corresponding orthogonal polynomials and linear forms interlace its zeros, see \cite{Ismail},  even though, as we will discuss now, the recursion matrix  $T_{JP}$ is not oscillatory.
The corresponding  monic recursion matrix $T_{JP}$  was considered in \cite[Section 4.3]{nuestro1} and we show that this recursion matrix was a positive matrix whenever the parameters $\alpha,\beta $ lay in the strip given by $|\alpha-\beta|<1$. 

\begin{lemma}[Jacobi--Piñeiro's recursion matrix bidiagonal factorization]
	The Jacobi--Piñeiro's recursion matrix has bidiagonal factorizations as in Equation \eqref{eq:bidiagonal}  with at least the following two set of parameters:
	{\scriptsize\begin{align*}
					&\begin{aligned}
							\alpha_{6n+1} &= \frac{( n + 1+\alpha ) (2 n + 1 + \alpha + \gamma) (2 n + 1 + \beta + \gamma)}{(3 n + 1 + \alpha + \gamma) (3 n + 2 + \alpha + \gamma) (3 n + 1 + \beta + \gamma)},\\	
							\alpha_{6n+2}  &= \frac{n (2 n + 1 + \gamma) (2 n + 1 + \alpha +  \gamma)}{(3 n + 2 + \alpha + \gamma) (3 n + 1 + \beta + \gamma) (3 n + 2 + \beta + \gamma)},\\	
							\alpha_{6n+3}&= \frac{(n + 1) (2 n + 1 + \gamma) (2 n + 2 + \beta + \gamma)}{(3 n + 2 + \alpha + \gamma) (3 n + 3 + \alpha +\gamma) (3 n + 2 + \beta + \gamma)},\\	
							\alpha_{6n+4} &=\frac {( n + 1+\beta) (2 n + 2 + \alpha + \gamma) (2 n + 2 + \beta + \gamma)}{(3 n + 3 + \alpha + \gamma) (3 n + 2 + \beta + \gamma) (3 n + 3 + \beta + \gamma)},\\	
							\alpha_{6n+5}&= \frac{(n + 1+\alpha-\beta ) (2 n + 2 + \gamma) (2 n + 2 + \alpha + \gamma)}{(3 n + 3 + \alpha + \gamma) (3 n + 4 + \alpha + \gamma) (3 n + 3 + \beta + \gamma)},\\
							\alpha_{6n+6}&= \frac{(n + 1-\alpha+\beta  ) (2 n + 2 +\gamma) (2 n + 3 + \beta + \gamma)}{(3 n + 4 + \alpha + \gamma) (3 n + 3 + \beta + \gamma) (3 n + 4 + \beta + \gamma)},
						\end{aligned}
					&&\begin{aligned}
							\tilde \alpha_{6n+1} &= \frac{( n + 1+\alpha ) (2 n + 1 + \alpha + \gamma) (2 n + 1 + \beta + \gamma)}{(3 n + 1 + \alpha + \gamma) (3 n + 2 + \alpha + \gamma) (3 n + 1 + \beta + \gamma)},\\	
							\tilde	\alpha_{6n+2}  &= \frac{(n-\alpha+\beta) (2 n + 1 + \gamma) (2 n + 1 + \beta +  \gamma)}{(3 n + 2 + \alpha + \gamma) (3 n + 1 + \beta + \gamma) (3 n + 2 + \beta + \gamma)},\\	
							\tilde	\alpha_{6n+3}&= \frac{(n + 1+\alpha-\beta) (2 n + 1 + \gamma) (2 n + 2 + \alpha + \gamma)}{(3 n + 2 + \alpha + \gamma) (3 n + 3 + \alpha +\gamma) (3 n + 2 + \beta + \gamma)},\\	
							\tilde	\alpha_{6n+4} &=\frac {( n + 1+\beta ) (2 n + 2 + \alpha + \gamma) (2 n + 2 + \beta + \gamma)}{(3 n + 3 + \alpha + \gamma) (3 n + 2 + \beta + \gamma) (3 n + 3 + \beta + \gamma)},\\	
							\tilde	\alpha_{6n+5}&= \frac{( n + 1) (2 n + 2 + \gamma) (2 n + 2 + \beta + \gamma)}{(3 n + 3 + \alpha + \gamma) (3 n + 4 + \alpha + \gamma) (3 n + 3 + \beta + \gamma)},\\
							\tilde	\alpha_{6n+6}&= \frac{( n + 1) (2 n + 2 +\gamma) (2 n + 3 + \alpha + \gamma)}{(3 n + 4 + \alpha + \gamma) (3 n + 3 + \beta + \gamma) (3 n + 4 + \beta + \gamma)},
						\end{aligned}
			\end{align*}}
	here  $n\in\N_0$.
\end{lemma}
\begin{proof}
We have $ \alpha_1 = c_0$ and  $\mathscr m_1$ and  $\alpha_4$ are gotten from
	$\mathscr m_1 \alpha_1 = b_1$ and 	$\mathscr m_1 + \alpha_4 = c_1$.
	Then,  $\mathscr l_n$, $\mathscr m_n$, $\alpha_{3(n-1)+1}$, $n = 2,3,\dots$, are determined recursively according to 
	$\mathscr l_n \alpha_{3(n-2)+1} = a_{n}$,
	$\mathscr l_n + \mathscr m_n \alpha_{3(n-1)+1} = b_n$
and
	$ m_n + \alpha_{3n+1} = c_n$
	(from the first relation we get  $\mathscr l_n$, from the second $m_n$  and from the third $\alpha_{3(n-1)+1}$).
	Now, these expressions for $\mathscr l$’s  and $\mathscr m$’s lead to the remaining $\alpha$’s.
Indeed, we have 
	$\alpha_2 + \alpha_3 = \mathscr m_1$ and	$ \alpha_5 \alpha_3 = \mathscr l_2$
	(we get $\alpha_3$ and  $\alpha_5$, respectively)
	and then we apply the recursion, $n \in\N$,
	$\alpha_{3n+2}+ \alpha_{3(n+1)} = \mathscr m_{n+1}$, 
	$\alpha_{3(n+1)+2}  \alpha_{3(n+1)} = \mathscr l_{n+2}$
	(in each iteration we obtain $\alpha_{3(n+1)}$ and  $\alpha_{3(n+1)+2}$, respectively).
\end{proof}
\begin{rem}
	The bidiagonal factorization $\{\tilde \alpha_n\}_{n=1}^\infty$ was found in \cite[Section 8.1]{Aptekarev_Kaliaguine_VanIseghem}, that is why we refer to it as the Aptekarev-Kalyagin-Van Iseghem (AKV) bidiagonal factorization.
\end{rem} 

For $n\in\N_0$,  given these two bidiagonal factorizations, the entries of the  corresponding lower unitriangular factor  $L=L_1L_2=\tilde L_1\tilde L_2$ of the lower factor $L$ in the Gauss--Borel factorization of the Jacobi--Piñeiro's Hessenberg transition matrix, can be expressed in the following two manners
\begin{align}\label{eq:lm_dos}
	\left\{\begin{aligned}
			\mathscr m_{2n+1}&=\alpha_{6n+2}+\alpha_{6n+3}=\tilde \alpha_{6n+2}+\tilde\alpha_{6n+3},&	\mathscr m_{2n+2}&=\alpha_{6n+5}+\alpha_{6n+6}=\tilde\alpha_{6n+5}+\tilde\alpha_{6n+6}, \\
			\mathscr l_{2n+2}&=\alpha_{6n+5}\alpha_{6n+3}=\tilde\alpha_{6n+5}\tilde\alpha_{6n+3},&	\mathscr l_{2n+3}&=\alpha_{6n+8}\alpha_{6n+6}=\tilde\alpha_{6n+8}\tilde\alpha_{6n+6}.
		\end{aligned}\right.
\end{align}

To better understand the dependence on the set of Jacobi--Piñeiro's parameters $(\alpha,\beta)$ we define some regions in the plane. 
Let us denote by $\mathscr R\coloneq\{(\alpha,\beta)\in\R^2, \alpha,\beta>-1, \alpha-\beta\not\in\Z \}$, that we call the natural region --where the orthogonality is well defined,  and divide it in the following  four regions:
\begin{align*}
	\mathscr R_1&\coloneq\{(\alpha,\beta )\in\mathscr R: \alpha-\beta>1\}, &	\mathscr R_2&\coloneq\{(\alpha,\beta )\in\mathscr R: 0<\alpha-\beta<1\}, \\ \mathscr R_3&\coloneq\{(\alpha,\beta )\in\mathscr R: -1<\alpha-\beta<0\}, &
	\mathscr R_4&\coloneq\{(\alpha,\beta )\in\mathscr R:  \alpha-\beta<-1\}.
\end{align*}
We show these regions in the following figure
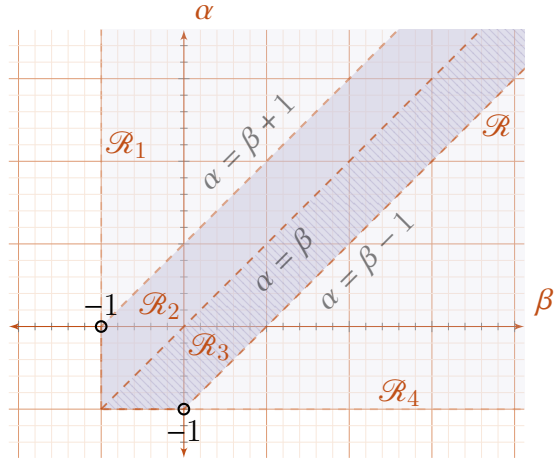
\begin{center}
	\begin{tikzpicture}[arrowmark/.style 2 args={decoration={markings,mark=at position #1 with \arrow{#2}}},scale=1]
			\begin{axis}[axis lines=middle,axis equal,grid=both,xmin=-1, xmax=3,ymin=-1.5, ymax=3.5,
					xticklabel,yticklabel,disabledatascaling,xlabel=$\beta$,ylabel=$\alpha$,every axis x label/.style={
							at={(ticklabel* cs:1)},
							anchor=south west,
						},
					every axis y label/.style={
							at={(ticklabel* cs:1.0)},
							anchor=south west,
						},grid style={line width=.1pt, draw=Bittersweet!10},
					major grid style={line width=.2pt,draw=Bittersweet!50},
					minor tick num=4,
					enlargelimits={abs=0.09},
					axis line style={latex'-latex'},Bittersweet]
					
					\draw [fill=DarkSlateBlue!20,opacity=.2,dashed,thick] (-1,4)--(-1,-1)--(5,-1)--(5,4)--(-1,4) ;
					\draw [fill=DarkSlateBlue!30,opacity=.5,dashed,thick] (-1,-1)--(-1,0)--(3,4) node[above, Black,sloped, pos=0.5] {$\alpha=\beta+1$}--(5,5)--(5,4) --(0,-1) node[below, Black,sloped, pos=0.6] {$\alpha=\beta-1$}--(-1,-1);
					\draw [
					pattern=north west lines, pattern color=DarkSlateBlue!50,opacity=.6,dashed] (-1,-1)--(0,-1)--(5,4)--(5,5)--(-1,-1) ;
					\draw [fill=DarkSlateBlue!30,opacity=.5,dashed,thick] (-1,-1) -- (4,4 )node[below, Black,sloped, pos=0.4] {$\alpha=\beta$};
					\draw[thick,black] (axis cs:-1,0) circle[radius=2pt,opacity=0.2,fill]node[left,above ] {$-1$} ;
					\draw[thick,black] (axis cs:0,-1)circle[radius=2pt,opacity=0.2,fill] node[right,below ] {$-1$} ;
					\node[anchor = north east,Bittersweet] at (axis cs: 4.1,2.7) {$\mathscr R$} ;
					\node[anchor = north east,Bittersweet] at (axis cs: -0.3,2.5) {$\mathscr R_1$} ;
					\node[anchor = north east,Bittersweet] at (axis cs: 0.1,.55) {$\mathscr R_2$} ;
					\node[anchor = north east,Bittersweet] at (axis cs: 0.7,0.05) {$\mathscr R_3$} ;
					\node[anchor = north east,Bittersweet] at (axis cs: 3,-.5) {$\mathscr R_4$} ;
				\end{axis}
		\end{tikzpicture}
\end{center}

\begin{lemma}\label{lemma:Regions_JP}
	\begin{enumerate}
			\item For the sequence $\{\alpha_n\}_{n\in\N}$ of the first bidiagonal factorization, we have
			\begin{enumerate}
					\item In the region $\mathscr R$  the sequence in TN but for $\alpha_5$ that is negative in $\mathscr R_4$ and $\alpha_6$ that is negative in $\mathscr R_1$.
					\item Is a TN sequence in the strip $\mathscr R_2\cup\mathscr R_3$. Excluding $\alpha_2=0$, the sequence is TP.
				\end{enumerate}
			\item For the AKV sequence $\{\tilde \alpha_n\}_{n\in\N}$, we have
			\begin{enumerate}
					\item In the region $\mathscr R$  the sequence in TP but for $\tilde \alpha_2$ that is negative in $\mathscr R_1\cup\mathscr R_2$,  $\tilde\alpha_8$ that is negative in $\mathscr R_1$ and $\tilde\alpha_3$ that is negative in $\mathscr R_4$.
					\item Is a TP sequence in the half strip $\mathscr R_3$. 
				\end{enumerate}
			
		\end{enumerate}
\end{lemma}
\begin{proof}
	For the first set of bidiagonal parameters $\{\alpha_n\}_{n\in\N}$, we check that all are positive in $\mathscr R$, but for $\alpha_2=0$ and $\alpha_5,\alpha_6$. From direct inspection we get that $\alpha_5<0$ when $1+\alpha-\beta<0$, i.e. in region $\mathscr R_4$ and $\alpha_6<0$ when $1-\alpha+\beta<0$, i.e. in region $\mathscr R_1$. Hence, the sequence  $\{\alpha_n\}_{n=1}^\infty$ is a TN sequence, TP but for $\alpha_2=0$, in region $\mathscr{R}_2\cup\mathscr R_3$, is TN in $\mathscr R$  but for $\alpha_5$ in $\mathscr R_4$ and TN in $\mathscr R$  but for $\alpha_6$ in $\mathscr R_1$. For theAKV  parameters $\{\tilde \alpha_n\}_{n\in\N}$, all are positive  in $\mathscr R$,  but for $\tilde\alpha_2,\tilde\alpha_8$ and $\tilde\alpha_3$. The entry $\tilde \alpha_2<0$ when $\alpha>\beta$, that is in $\mathscr R_1\cup \mathscr R_2$, $\tilde \alpha_8<0$ when $1-\alpha+\beta<0$ i.e. in $\mathscr R_1$ and $\tilde\alpha_3$ when $1+\alpha-\beta<0$, i.e. in $\mathscr R_4$. 
\end{proof}

\begin{lemma}
	For the two first subdiagonals of the lower triangular matrix $L$ in the Gauss--Borel factorization of the  Jacobi--Piñeiro's Hessenberg recursion matrix  $T_{JP}$ we have
	\begin{enumerate}
			\item The sequence $\{\mathscr m_n\}_{n=1}^\infty$ is TP in the definition region  $\mathscr R$. 
			\item The sequence $\{\mathscr l_n\}_{n=1}^\infty$  is TP but for  $\mathscr l_2$ ($\mathscr l_3$),  that is negative in $\mathscr R_4$ ($\mathscr R_1$).
		\end{enumerate}
\end{lemma}
\begin{proof}
	From \eqref{eq:lm_dos} and the first bidiagonal factorization we get that the lower triangular $L$ has all its entries in the two first subdiagonals TP but for  $\mathscr l_2$ ($\mathscr l_3$),  that is negative in $\mathscr R_4$ ($\mathscr R_1$), and maybe $\mathscr m_5=\alpha_5+\alpha_6$. Looking now at the AKV factorization we see that $\mathscr m_5=\tilde \alpha_5+\tilde\alpha_6>0$.
\end{proof}

\begin{pro}
	The Jacobi--Piñeiro's  recursion matrix satisfies:
	\begin{enumerate}
			\item 	 Is oscillatory if and only if the parameters $(\alpha,\beta)$ belong to the  strip $\mathscr R_2\cup\mathscr R_3$.
			\item    Admits a positive bidiagonal factorization at least if the parameters belong to the lower half strip $\mathscr R_3$.
			\item The retraction of the complementary matrix $T_{JP}^{(4)}$ ( $T_{JP}^{(5)}$), described in \cite[Theorem 11]{previo2}
			 is   oscillatory in the region $\mathscr R_2\cup \mathscr R_3\cup\mathscr R_4$ ($\mathscr R$).
		\end{enumerate}
\end{pro}
\begin{proof}
	\begin{enumerate}
			\item It follows from  Lemma \ref{lemma:Regions_JP}.
			\item The AKV bidiagonal factorization sequence is TP in $\mathscr R_3$.
			\item We use the Gauss--Borel factorization of these retractions described in \cite[Theorem 11]{previo2}, that we know have a bidiagonal factorization with TP sequences.
		\end{enumerate}
\end{proof}

\begin{rem}
	From the previous discussion of  the  Jacobi--Piñeiro's recursion matrix it becomes clear that demanding the matrix to  have a positive bidiagonal factorization  is sufficient but not necessary to have spectral measures. In the natural region $\mathscr R$ the Jacobi--Piñeiro's weights exist, are positive with support on $[0,1]$.  However, we know that it is an oscillatory matrix only   in the strip $\mathscr R_2\cup\mathscr R_3$. The associated matrix that is regular oscillatory in the natural region $\mathscr R$ is the retracted complementary matrix $T_{JP}^{(5)}$. This observation leads to the question: Is it enough to have a spectral Favard theorem and positive measures that a retracted complementary matrix of the banded Hessenberg matrix is oscillatory?
\end{rem}

\section{Applications to Darboux  transformations}\label{S:Darboux}

\subsection{Darboux transformations of  oscillatory banded Hessenberg matrices}
We now show how our construction connects with those of the seminal paper \cite{Aptekarev_Kaliaguine_VanIseghem} by Aptekarev, Kalyagin and Van Iseghem on genetic sums, vector convergents, Hermite--Padé approximants and and Stieltjes problems, and with the Darboux--Christoffel transformations discussed in \cite{bfm}.  We identify the Darboux transformations of the oscillatory banded Hessenberg matrices with Christoffel transformations of the spectral measures.
Recall that $\nu$ is the  initial condition  given in Definition \ref{def:typeI}.

\begin{defi}[Darboux transformed Hessenberg matrices]
	Given an oscillatory banded lower Hessenberg matrix $T$  and its bidiagonal factorization as in \eqref{eq:bidiagonal}, we consider the semi-infinite matrices
\begin{align*}
	\hat T&\coloneq L_2 U L_1, & \hat {\hat T} &\coloneq U L_1 L_2.
\end{align*}
We will refer to these matrices as the first and second Darboux transformations of the  banded Hessenberg matrix $T$.
\end{defi}
\begin{rem}\label{rem:hat T is not tilde T}
		These auxiliary matrices $\hat T^{[N]}$ are not the $m$-th leading principal submatrix  of the matrix $\hat T\coloneq L_2UL_1$. The difference is in the last diagonal entry.
	The entries of $\hat T$  are
	\begin{align}\label{eq:entries hat_T}
	&\left\{\begin{aligned}
	\hat c_n & = \alpha_{3n+2} + \alpha_{3n+1} + \alpha_{3n}  , \\ 
	\hat b_n & = \alpha_{3n}  \alpha_{3n-1}+ \alpha_{3n+1} \alpha_{3n-1}+ \alpha_{3n}\alpha_{3n-2} , \\
	\hat a_n & =\alpha_{3n}\alpha_{3n-2} \alpha_{3n-4} ,
	\end{aligned}\right.&
	\end{align}
	All the entries of the $(N+1)$-th leading principal submatrix   of $\,\hat T$ coincide with those of $\hat T^{[N]}$ but for the last diagonal entry,  as
	$(\hat T^{[N]})_{N+1,N+1}=\alpha_{3N}+\alpha_{3N+1}$ while $\hat c_{N+1}=\alpha_{3N}+\alpha_{3N+1}+\alpha_{3N+2}$.
	\end{rem}
\begin{rem} \begin{enumerate}
		\item From definition it is immediately checked that both Darboux  transformed Hessenberg matrices has the same banded structure as $T$.
\item	For
 coefficients of the second Darboux transform $\hat {\hat T} $ we have 
\begin{align*}
\left\{\begin{aligned}
		\hat {\hat c}_n& = \alpha_{3n+3} + \alpha_{3n+2} + \alpha_{3n+1}, \\
	\hat {\hat b}_n& = \alpha_{3n+1}  \alpha_{3n}+ \alpha_{3n+2} \alpha_{3n}+ \alpha_{3n+1}\alpha_{3n-1}  , \\
	\hat {\hat a}_n &=\alpha_{3n+1}\alpha_{3n-1} \alpha_{3n-3}.
\end{aligned}\right.
\end{align*}
\item If $\,T$ has a positive bidiagonal factorization, so that we can take $\alpha_2>0$, then the Darboux transforms $\hat T, \hat{\hat T}$ are   oscillatory.
	\end{enumerate}
\end{rem}
Associated with these Hessenberg matrices we introduce the vectors of polynomials
\begin{align}\label{eq:L1hatB}
	\hat{\hat B}&\coloneq UB, & \hat B&=L_2\hat{\hat B}\coloneq L_2 U B.
\end{align}
Notice that  $\hat B_n$ and $ \hat {\hat B}_n$ are monic with
$\deg \hat B_n=	\deg \hat {\hat B}_n= n+1$. 
\begin{lemma}
We have
	\begin{align}\label{eq:L1hatBbis}
		L_1 \hat B=xB.
	\end{align}
\end{lemma}
\begin{proof}
	Equations \eqref{eq:L1hatB} imply
	$	L_1 \hat B=L_1L_2 U B=TB=xB$.
\end{proof}
\begin{pro}
The eigenvalue properties
$	\hat T \hat B=x \hat B$ and $\hat{\hat T}\hat{\hat B}= x\hat{\hat B}$ are satisfied.
\end{pro}
\begin{proof}

Equations \eqref{eq:L1hatB} and \eqref{eq:L1hatBbis} lead by direct computation  to
\begin{align*}
	\hat T \hat B&=L_2 U L_1\hat B=x  L_2 U B=x \hat B,&
	\hat{\hat T}\hat{\hat B}&= U L_1 L_2 UB=UTB=xUB=x\hat{\hat B}.
\end{align*}
\end{proof}

Let us denote by  if $\tilde T^{[n]}$ and $\tilde{\tilde T}^{[n]}$  the 
$(n+1)$-th leading principal submatrices of  $\hat T$ and $\hat{\hat T}$.

\begin{lemma}
	The polynomials $\hat B_n$ and $\hat {\hat B}_n$ 
	can be expressed  as
	$		\hat{B}_n=x\tilde B_n$ and $ \hat{\hat B}_n=x \tilde{\tilde B}_n$, 
	with the monic polynomials  $\tilde B_n, \tilde{\tilde B}_n$ having degree $n$. 
\end{lemma}
\begin{proof}
	One has that
$\hat{\hat B}_0=\hat B_0=\alpha_1+B_1=\alpha_1+x-c_0=x$.
Then, as the sequences of polynomials are found by the recurrence determined by the banded Hessenberg matrices $\hat T$ and $\hat{\hat T}$, respectively,  we find that the desired result.
\end{proof}

We call these   polynomials $\tilde B_n$ and $  \tilde{\tilde B}_n$ as Darboux transformed polynomials of type II.

\begin{pro}
The entries of  the  Darboux transformed polynomial sequences of type II
	\begin{align}\label{eq:hatB}
	 \tilde  B&=\frac{1}{x} L_2 U B, &	\tilde{\tilde B}&\coloneq \frac{1}{x}UB,
	\end{align}
read
\begin{align}
\label{eq:tildeB}	\tilde B_n &
=\frac{1}{x}\big(B_{n+1}+(\alpha_{3n+1}+\alpha_{3n})B_n+\alpha_{3n}\alpha_{3n-2} B_{n-1}\big), 
&
\tilde{	\tilde B}_n&
=\frac{1}{x}\big(B_{n+1}+\alpha_{3n+1}B_{n}\big)
\end{align}
were  we take $\alpha_k=0$ for $k\in \Z_-$. The following determinantal expressions hold
\begin{align*}
	\tilde B_{n+1}&=\det \big(xI_{n+1}-\tilde T^{[n]}\big), &  \tilde{	\tilde B}_{n+1}&=\det \big(xI_{n+1}-\tilde{	\tilde T}^{[n]}\big).
\end{align*}

\end{pro}
\begin{proof}
	
Equation  \eqref{eq:tildeB} 
appears as the entries of the defining equations. The determinantal expressions follow from the fact that its expansions along the last row satisfy the recursion relations with adequate initial conditions.
\end{proof}

Following definitions given in \eqref{eq:TNk} and \eqref{eq:BNk} we consider similar objects in this context. That is, we denote by $\tilde T^{[n,k]}$ ($\tilde {\tilde T}^{[n,k]}$)  the matrix obtained from $\tilde T^{[n]}$ ($\tilde {\tilde T}^{[n]}$) by erasing the first $k$ rows and columns. The corresponding characteristic polynomials are 
\begin{align*}
	\tilde B^{[k]}_{n+1}&=\det \big(xI_{n+1-k}-\tilde T^{[n,k]}\big), &  \tilde{	\tilde B}^{[k]}_{n+1}&=\det \big(xI_{n+1-k}
	-\tilde{	\tilde T}^{[n,k]}\big).
\end{align*}
These polynomials $\tilde B_{n}^{[k]}$ ($\tilde {\tilde B}_n^{[k]}$)
satisfy the same recursion relations, determined by $\tilde T$ ($\tilde {\tilde T}$) as do $\tilde B_{n}$ ($\tilde {\tilde B}_n$) but with different initial conditions. Following ii) in Proposition \ref{pro:determintal_second_kind} we have the transformed recursion polynomials of type II 
\begin{align*}
\tilde	B_{n+1}^{(1)}&=\tilde B^{[1]}_{n+1},& \tilde{\tilde	B}_{n+1}^{(1)}&=\tilde{\tilde B}^{[1]}_{n+1}, &
 \tilde	B_{n+1}^{(2)}&= \tilde B^{[2]}_{n+1}-\nu\tilde B^{[1]}_{n+1}, &
 \tilde{\tilde B}_{n+1}^{(2)}&= \tilde{	\tilde B}^{[2]}_{n+1}-\nu \tilde{\tilde B}^{[1]}_{n+1}.
\end{align*}
Then we consider the following vectors of polynomials 
\begin{align*}
	\hat	B_{n+1}^{(1)}&=	x\tilde	B_{n+1}^{(1)},& \hat{\hat	B}_{n+1}^{(1)}&=x\tilde{\tilde B}^{[1]}_{n+1}, &
	\hat 	B_{n+1}^{(2)}&= x \tilde B^{(2)}_{n+1},&
	\hat 	{\hat B}_{n+1}^{(2)}&= x \tilde{\tilde B}^{(2)}_{n+1}.
\end{align*}

\begin{pro}[Vector Convergents]
	These recursion polynomials correspond to the  vector convergent  $y^1_n=(A_{n,0},A_{n,1},A_{n,2})$ discussed in \cite{Aptekarev_Kaliaguine_VanIseghem} as follows
	\begin{align*}
		B_n&=A_{3n,0}, & \hat B_n&=A_{3n+1,0}, &\hat {\hat B}_n&=A_{3n+2,0},\\
		B_n^{(1)}&=A_{3n,1}, & \hat B_n^{(1)}&=A_{3n+1,1}, &\hat {\hat B}_n^{(1)}&=A_{3n+2,1},\\
-\frac{1}{\nu}	B_n^{(2)}&=A_{3n,2}, & -\frac{1}{\nu}	\hat B_n^{(2)}&=A_{3n+1,2}, &-\frac{1}{\nu}	\hat {\hat B}_n^{(2)}&=A_{3n+2,2},
	\end{align*}
\end{pro}
\begin{proof}
		It follows from the fact that they satisfy the recursion relation \cite[Equation (23)]{Aptekarev_Kaliaguine_VanIseghem} and adequate initial conditions.
\end{proof}

Then, following this dictionary  the important \cite[Lemma 5]{Aptekarev_Kaliaguine_VanIseghem} states for $x\geqslant 0$ that
\begin{align*}
	\begin{vNiceMatrix}
		\hat B_n&  \hat B^{(1)}_n \\[2pt]
		 B_n &   B^{(1)}_n 
	\end{vNiceMatrix}&\leqslant 0, & 	
\begin{vNiceMatrix}
	 {\hat B}_n &   {\hat B}^{(2)}_n \\[2pt]
	B_n &   B^{(2)}_n 
\end{vNiceMatrix}&\leqslant 0, &\begin{vNiceMatrix}
\hat{\hat B}_n &  \hat{\hat B}^{(1)}_n \\[2pt]
B_n &   B^{(1)}_n 
\end{vNiceMatrix}&\leqslant 0, & 	
\begin{vNiceMatrix}
\hat{\hat B}_n &   \hat{\hat B}^{(2)}_n \\[2pt]
B_n &   B^{(2)}_n 
\end{vNiceMatrix}&\leqslant 0, \\
\begin{vNiceMatrix}
	\hat{\hat B}_n &  \hat{\hat B}^{(1)}_n \\[2pt]
	\hat B_n &   \hat B^{(1)}_n 
\end{vNiceMatrix}&\leqslant 0, & 	
\begin{vNiceMatrix}
	\hat{\hat B}_n &   \hat{\hat B}^{(2)}_n \\[2pt]
\hat 	B_n &   \hat B^{(2)}_n 
\end{vNiceMatrix}&\leqslant 0,  &\begin{vNiceMatrix}
{ B}_{n+1} &  { B}^{(1)}_{n+1} \\[2pt]
\hat B_n &   \hat B^{(1)}_n 
\end{vNiceMatrix}&\leqslant 0, & 	
\begin{vNiceMatrix}
{ B}_{n+1} &  { B}^{(2)}_{n+1} \\[2pt]
\hat 	B_n &   \hat B^{(2)}_n 
\end{vNiceMatrix}&\leqslant 0,  \\
\begin{vNiceMatrix}
	{ B}_{n+1} &  { B}^{(1)}_{n+1} \\[2pt]
\hat{\hat B}_n&   \hat{\hat B}^{(1)}_n 
\end{vNiceMatrix}&\leqslant 0, & 	
\begin{vNiceMatrix}
{ B}_{n+1} &  { B}^{(2)}_{n+1} \\[2pt]
	\hat{\hat B}_n &   \hat{\hat B}^{(2)}_n 
\end{vNiceMatrix}&\leqslant 0,  &\begin{vNiceMatrix}
	\hat{ B}_{n+1} &  \hat{ B}^{(1)}_{n+1} \\[2pt]
	\hat{\hat B}_n &   \hat{\hat B}^{(1)}_n 
\end{vNiceMatrix}&\leqslant 0, & 	
\begin{vNiceMatrix}
\hat	{ B}_{n+1} &  \hat{ B}^{(2)}_{n+1} \\[2pt]
	\hat 	{\hat B}_n &   \hat {\hat B}^{(2)}_n 
\end{vNiceMatrix}&\leqslant 0.
\end{align*}
\begin{rem}
	Using these facts,  Aptekarev, Kalyagin and Van Iseghem in \cite[Lemmata 6 \& 7]{Aptekarev_Kaliaguine_VanIseghem}   deduce the degree of polynomials, simplicity of zeros and interlacing properties of $B_n$ with $B_{n-1}$,  $B_n^{(1)}$ and $B^{(2)}_n$.
Notice that we derive the same result by just using the spectral properties of regular oscillatory matrices.
\end{rem}

For recursion polynomials of type I, we introduce the following polynomials
$	\hat A^{(2)}\coloneq A^{(1)}L_1$,	$\hat A^{(1)}\coloneq A^{(2)}L_1$,
$\hat{	\hat A}^{(1)}\coloneq A^{(1)}L_1L_2$ and 	$\hat {\hat A}^{(2)}\coloneq A^{(2)}L_1L_2$.
\begin{pro}
Vectors  $	\hat A^{(1)}$, $	\hat A^{(2)}$ are left eigenvectors of $\hat T$ and $\hat {	\hat A}^{(1)}$, $\hat {	\hat A}^{(2)}$ are left eigenvectors of $\hat{\hat T}$.
\end{pro}
\begin{proof}
A direct computation shows that
$		\hat A^{(2)}\hat T=A^{(1)}L_1 L_2UL_1= A^{(1)}T L_1=xA^{(1)}L_1=	x \hat A^{(2)}$.
The other cases are proven similarly.
\end{proof}
\begin{lemma}\label{lemma:previous}
Let us assume that  $1+\alpha_{2}\nu=0$. Then, 
	$\hat A^{(2)}_0=0$ and $\hat A^{(2)}_1=\frac{1}{\alpha_3\alpha_1}x$.	
\end{lemma}
\begin{proof}
Let us consider the vector
$	\hat A^{(2)}=A^{(1)} L_1$
with  components
$	\hat A^{(2)}_n=A^{(1)}_n+\alpha_{3n+2}A^{(1)}_{n+1}$, $n\in\N_0$.
The first two entries are
\begin{align*}
	\hat A^{(2)}_0&=A^{(1)}_0+\alpha_{2}A^{(1)}_{1}=1+\alpha_{2}\nu,\\
	\hat A^{(2)}_1&=A^{(1)}_1+\alpha_{5}A^{(1)}_{2}=\nu-\alpha_5\frac{c_0+b_1\nu}{a_2}+\frac{\alpha_5}{a_2}x=\nu-\alpha_5\frac{\alpha_1+(\alpha_3+\alpha_2)\alpha_1\nu}{\alpha_5\alpha_3\alpha_1}+\frac{\alpha_5}{\alpha_5\alpha_3\alpha_1}x\\&=-\frac{1+\alpha_2\nu}{\alpha_3}
	+\frac{1}{\alpha_3\alpha_1}x.
\end{align*}
Here we have used Definition \ref{def:typeI}  and
$c_0=\alpha_1$, $b_1=(\alpha_3+\alpha_2)\alpha_1$ and $a_2=\alpha_5\alpha_3\alpha_1$.
Then, as $1+\alpha_{2}\nu=0$, we find
the stated result.
\end{proof}

\begin{pro}
	If $	1+\nu\alpha_2=0$, 
we can write
		$\hat A^{(2)}_n= x \tilde A^{(2)}_n$ and $\hat {\hat A}^{(1)}_n= x \tilde {\tilde A}^{(1)}_n$, 
for some polynomials $ \tilde A^{(2)}_n,\tilde {\tilde A}^{(1)}_n$.
\end{pro}
\begin{proof}
The recursion relation
$\hat c_0
\hat A^{(2)}_0+\hat b_1\hat A^{(2)}_1+\hat a_2 \hat A^{(2)}_2=x\hat A^{(2)}_0$
and Lemma \ref{lemma:previous} gives   $\hat A^{(2)}_2=-\frac{\hat b_1}{\hat a_2 \alpha_3\alpha_1}x$.  Hence,   induction leads  to the conclusion that $\hat A^{(2)}_n=x\tilde A^{(2)}_n$, for some polynomial $\tilde A^{(2)}_n$.
\end{proof}
%
%
%
%
%
\begin{lemma}
	We have
$\hat {\hat {A}}^{(2)}_0=0$ and $\hat {\hat {A}}^{(2)}_1=\frac{1}{\alpha_4}x$.
\end{lemma}
\begin{proof}
Let us consider the vector
$	\hat A^{(1)}=A^{(2)} L_1$
with  components
$	\hat A^{(1)}_n=A^{(2)}_n+\alpha_{3n+2}A^{(2)}_{n+1}$, $n\in\N_0$.
The first three entries are
\begin{align*}
	\hat A^{(1)}_0&=A^{(2)}_0+\alpha_{2}A^{(2)}_{1}=\alpha_{2},\\
	\hat A^{(1)}_1&=A^{(2)}_1+\alpha_{5}A^{(2)}_{2}=1-\alpha_5\frac{b_1}{a_2}=1-\alpha_5\frac{(\alpha_3+\alpha_2)\alpha_1}{\alpha_5\alpha_3\alpha_1}=-\frac{\alpha_2}{\alpha_3},
		\\
		\hat A^{(1)}_2&=A^{(2)}_2+\alpha_{8}A^{(2)}_{3}=-\frac{b_1}{a_2}+\frac{\alpha_8}{a_3}	\Big(b_2\frac{b_1}{a_2}+x-c_1\Big)		=\frac{b_1}{a_2}\Big(\frac{\alpha_8b_2}{a_3}-1\Big)-\frac{\alpha_8c_1}{a_3}+\frac{\alpha_8}{a_3}x\\&=
		\frac{\alpha_3+\alpha_2}{\alpha_5\alpha_3}\Big(
		\frac{\alpha_6\alpha_4+\alpha_5\alpha_4+\alpha_5\alpha_3}{\alpha_6\alpha_4}-1
		\Big)-\frac{\alpha_4+\alpha_3+\alpha_2}{\alpha_6\alpha_4}+\frac{1}{\alpha_6\alpha_4}x\\&=
		\frac{\alpha_3+\alpha_2}{\alpha_5\alpha_3}
		\frac{\alpha_5\alpha_4+\alpha_5\alpha_3}{\alpha_6\alpha_4}
		-\frac{\alpha_4+\alpha_3+\alpha_2}{\alpha_6\alpha_4}+\frac{1}{\alpha_6\alpha_4}x=\frac{1}{\alpha_6\alpha_4}\Big(
		\frac{\alpha_3+\alpha_2}{\alpha_3}
		(\alpha_4+\alpha_3)
		-\alpha_4-\alpha_3-\alpha_2+x\Big)\\&=
		\frac{\alpha_2}{\alpha_6\alpha_3}+\frac{1}{\alpha_6\alpha_4}x.
\end{align*}
Then, we consider $	\hat{ \hat {A}}^{(2)}=\hat {A}^{(1)} L_2$ with  components
$\hat {\hat A}^{(2)}_n=\hat A^{(1)}_n+\alpha_{3n+3}\hat A^{(1)}_{n+1}$, $n\in\N_0$.
The first two components being
\begin{align*}
\hat {\hat A}^{(2)}_0&=\hat A^{(1)}_0+\alpha_{3}\hat A^{(1)}_{1}=\alpha_2+\alpha_3\Big(-\frac{\alpha_2}{\alpha_3}\Big)=0,&
\hat {\hat A}^{(2)}_1&=\hat A^{(1)}_1+\alpha_{6}\hat A^{(1)}_{2}=-\frac{\alpha_2}{\alpha_3}+\alpha_6\Big(	\frac{\alpha_2}{\alpha_6\alpha_3}+\frac{1}{\alpha_6\alpha_4}x\Big)=\frac{1}{\alpha_4}x,
\end{align*}
and the result follows.
\end{proof}
\begin{pro}
	There are polynomials $\tilde{\tilde A}^{(2)}_n $ such that $\hat{\hat  A}^{(2)}_n  =x\tilde{\tilde A}^{(2)}_n $. 
\end{pro}
\begin{proof}
	It holds for the two first entries $\hat {\hat A}^{(2)}_0$ and $\hat {\hat A}^{(2)}_1$. Hence, from the recursion relation $\hat{\hat A}^{(2)}\hat {\hat T}= x\hat{\hat A}^{(2)}$ we get that it holds for any natural number $n$.
\end{proof}

We name the  polynomials $\hat A^{(1)}_n$, $\tilde{\tilde A}^{(1)}_n $, ${\tilde A}^{(2)}_n$ and $\tilde{\tilde A}^{(2)}_n $ as Darboux transformed polynomials of type I. 

\begin{pro}
Let us assume that $1+\nu\alpha_2=0$. The entries of the Darboux transformed  polynomials sequences of type I 
\begin{align*}
	\tilde A^{(2)}&=\frac{1}{x} A^{(1)}L_1, & \tilde{\tilde A}^{(1)}&=\frac{1}{x} A^{(1)}L_1L_2, & 	\hat  A^{(1)}&= A^{(2)}L_1, & \tilde{\tilde A}^{(2)}&=\frac{1}{x} A^{(2)}L_1L_2
\end{align*}
are given by
\begin{align*}
	\tilde A^{(2)}_n&=\frac{1}{x}\big(A^{(1)}_n+\alpha_{3n+2}A^{(1)}_{n+1}\big), & 	
\tilde 	{\tilde A}^{(1)}_n
&=\frac{1}{x}\big(A^{(1)}_n+(\alpha_{3n+2}+\alpha_{3n+3})A^{(1)}_{n+1}+\alpha_{3n+5}\alpha_{3n+3}A^{(1)}_{n+2}\big), \\
	\hat A^{(1)}_n&=A^{(2)}_n+\alpha_{3n+2}A^{(2)}_{n+1}, & 	
\tilde 	{\tilde A}^{(2)}_n&=\frac{1}{x}\big(A^{(2)}_n+(\alpha_{3n+2}+\alpha_{3n+3})A^{(2)}_{n+1}+\alpha_{3n+5}\alpha_{3n+3}A^{(2)}_{n+2}\big).
\end{align*}
\end{pro}

  \subsection{Spectral representation and Christoffel transformations}
  
 We identify the entries in the bidiagonal factorization \eqref{eq:bidiagonal} with simple rational expressions in terms of the recursion
  polynomials valuated at the origin.
\begin{teo}[Parametrization of the bidiagonal factorization]\label{teo:alphasBA}
	The $\alpha$'s in the bidiagonal factorization \eqref{eq:bidiagonal} can be expressed in terms of the recursion polynomials evaluated at $x=0$ as follows:
	\begin{align}
		\alpha_{3n+1}&=-\frac{B_{n+1}(0)}{B_{n}(0)},\label{eq:alphamod1}\\
		\alpha_{3n+2}&=-\frac{A^{(1)}_{n}(0)}{A^{(1)}_{n+1}(0)}, &\text{$1+\nu\alpha_2=0$ is required,}\label{eq:alphamod2}\\
\label{eq:alphamod3}
	\alpha_{3n+3}&
=
-\frac{A^{(1)}_{n}(0)A^{(2)}_{n+1}(0)-A^{(1)}_{n+1}(0)A^{(2)}_{n}(0)}{A^{(1)}_{n+1}(0) A^{(2)}_{n+2}(0)-A^{(1)}_{n+2}(0)A^{(2)}_{n+1}(0)}
	\frac{A^{(1)}_{n+2}(0)}{A^{(1)}_{n+1}(0)}.
\end{align}
The relations
\begin{align}\label{eq:quasi-det_alpha}
	\begin{bNiceMatrix}
		\alpha_{3n+2}+\alpha_{3n+3} & \alpha_{3n+5}\alpha_{3n+3}
	\end{bNiceMatrix}=-	\begin{bNiceMatrix}
		A^{(1)}_{n}(0) & 	A^{(2)}_{n}(0) 
	\end{bNiceMatrix}	\begin{bNiceMatrix}
		A^{(1)}_{n+1}(0) & A^{(2)}_{n+1}(0)\\[2pt]
		A^{(1)}_{n+2}(0) &A^{(2)}_{n+2}(0)
	\end{bNiceMatrix}^{-1}
\end{align}
are satisfied as well.
\end{teo}
\begin{proof}
Equation \eqref{eq:L1hatB} and the fact that $\hat B(0)=0$ gives that $UB(0)=0$. Hence, we get $\alpha_{3n+1}B_n(0)+B_{n+1}(0)=0$ and \eqref{eq:alphamod1} follow. Now, as $	\hat A^{(1)}= A^{(1)}L_1$ and $\hat A^{(1)}(0)=0$ implies $A^{(1)}(0)L_1=0$. Hence, $A_{n}^{(1)}(0)+A_{n+1}^{(1)}(0)\alpha_{3n+2}=0$ and we find \eqref{eq:alphamod2}.
To prove \eqref{eq:quasi-det_alpha} we  observe that 
\begin{align*}
	A^{(1)}_n(0)+(\alpha_{3n+2}+\alpha_{3n+3})A^{(1)}_{n+1}(0)+\alpha_{3n+5}\alpha_{3n+3}A^{(1)}_{n+2}(0)&=0,\\
		A^{(2)}_n(0)+(\alpha_{3n+2}+\alpha_{3n+3})A^{(2)}_{n+1}(0)+\alpha_{3n+5}\alpha_{3n+3}A^{(2)}_{n+2}(0)&=0,
\end{align*}
so that
\begin{align*}
\begin{bNiceMatrix}
	A^{(1)}_{n}(0) & 	A^{(2)}_{n}(0) 
		\end{bNiceMatrix}
	+	\begin{bNiceMatrix}
		\alpha_{3n+2}+\alpha_{3n+3} & \alpha_{3n+5}\alpha_{3n+3}
	\end{bNiceMatrix}
\begin{bNiceMatrix}
	A^{(1)}_{n+1}(0) & A^{(2)}_{n+1}(0)\\[2pt]
	A^{(1)}_{n+2}(0) &A^{(2)}_{n+2}(0)
	\end{bNiceMatrix}=0,
\end{align*}
and Equation \eqref{eq:quasi-det_alpha} follows.

This equation implies component-wise the following relations
\begin{align*}
	\alpha_{3n+2}+\alpha_{3n+3} &=-\frac{A^{(1)}_{n}(0)A^{(2)}_{n+2}(0)-A^{(1)}_{n+2}(0)A^{(2)}_{n}(0)}{A^{(1)}_{n+1}(0) A^{(2)}_{n+2}(0)-A^{(1)}_{n+2}(0)A^{(2)}_{n+1}(0)}, &
	\alpha_{3n+5}\alpha_{3n+3}&=-\frac{A^{(1)}_{n+1}(0)A^{(2)}_{n}(0)-A^{(1)}_{n}(0)A^{(2)}_{n+1}(0)}{A^{(1)}_{n+1}(0) A^{(2)}_{n+2}(0)-A^{(1)}_{n+2}(0)A^{(2)}_{n+1}(0)}.
	\end{align*}
Thus, we get Equation \eqref{eq:alphamod3}.
\end{proof}

With the previous identification we are ready to show the complete correspondence of the described Darboux transformations of the oscillatory banded Hessenberg matrix $T$ with Christoffel perturbations of the corresponding pair of positive Lebesgue--Stieltjes measures $(\d\psi_1,\d\psi_2)$. 

	\begin{teo}[Darboux vs Christoffel transformations] \label{theorem:3} For $\alpha_2=-\frac{1}{\nu}>0$,  the multiple orthogonal polynomial sequences $\big\{{	\tilde B}_n,\hat A^{(1)}_n,\tilde A^{(2)}_n\big\}_{n=0}^\infty$ and  $\big\{\tilde{	\tilde B}_n,\tilde{	\tilde A}^{(1)}_n,
\tilde{	\tilde A}^{(2)}_n\big\}_{n=0}^\infty$ correspond to the Christoffel transformations given in  \cite[Theorems 4 \& 6]{bfm} of the multiple orthogonal polynomial sequence  $\{ B_n,  A^{(1)}_n, A^{(2)}_n\}_{n=0}^\infty$. If the original couple of Lebesgue--Stieltjes measures is $(\d\psi_1,\d\psi_2)$, then  the corresponding transformed pairs of measures are $(\d\psi_2,x\d\psi_1)$ and $(x\d\psi_1,x\d\psi_2)$, respectively. 
\end{teo}

\begin{proof}
Recalling \eqref{eq:alphamod2}, that $c_n=\alpha_{3n+1}+\alpha_{3n}+\alpha_{3n-1}$ and that $a_{n+1}=\alpha_{3n+2}\alpha_{3n}\alpha_{3n-2}$ we write
\begin{align*}
\alpha_{3n+1}+\alpha_{3n}&=\frac{A^{(1)}_{n-1}(0)}{A^{(1)}_{n}(0)}+c_n, &
\alpha_{3n}\alpha_{3n-2}&=-\frac{A^{(1)}_{n+1}(0)}{A^{(1)}_{n}(0)} a_{n+1}.
\end{align*}
Then, using Theorem \ref{teo:alphasBA} and the first equation in  \eqref{eq:tildeB} we get
\begin{align*}
	\tilde B_n &
	=\frac{1}{x}\bigg(B_{n+1}+\Big(\frac{A^{(1)}_{n-1}(0)}{A^{(1)}_{n}(0)}+c_n\Big)B_n-\frac{A^{(1)}_{n+1}(0)}{A^{(1)}_{n}(0)} a_{n+1}B_{n-1}\bigg), \\
		\hat A^{(1)}_n&=A^{(2)}_n-\frac{A^{(1)}_{n}(0)}{A^{(1)}_{n+1}(0)}A^{(2)}_{n+1},\\
			\tilde A^{(2)}_n&=\frac{1}{x}\bigg(A^{(1)}_n-\frac{A^{(1)}_{n}(0)}{A^{(1)}_{n+1}(0)}A^{(1)}_{n+1}\bigg).
\end{align*}
These three equations are the Christoffel formulas in \cite[Theorem 4]{bfm} for the permuting Christoffel transformation
$(\d\psi_1,\d\psi_2)\to (\d\psi_2,x\d\psi_1)$.
Also, using again Theorem \ref{teo:alphasBA},  the second equation in  \eqref{eq:tildeB} and \eqref{eq:quasi-det_alpha} we get
\begin{align*}
	\tilde{	\tilde B}_n&=\frac{1}{x}\bigg(B_n-\frac{B_{n+1}(0)}{B_{n}(0)}B_{n+1}\bigg),\\
	\tilde{	\tilde A}^{(1)}_n&=\frac{1}{x}\bigg(A^{(1)}_n-	\begin{bNiceMatrix}
		A^{(1)}_{n}(0) & 	A^{(2)}_{n}(0) 
	\end{bNiceMatrix}	\begin{bNiceMatrix}
		A^{(1)}_{n+1}(0) & A^{(1)}_{n+2}(0)\\[2pt]
		A^{(2)}_{n+1}(0) &A^{(2)}_{n+2}(0)
	\end{bNiceMatrix}^{-1}\begin{bNiceMatrix}
A^{(1)}_{n+1}\\[2pt]A^{(1)}_{n+2}
\end{bNiceMatrix}\bigg),\\
		\tilde{	\tilde A}^{(2)}_n&=\frac{1}{x}\bigg(A^{(2)}_n-	\begin{bNiceMatrix}
			A^{(1)}_{n}(0) & 	A^{(2)}_{n}(0) 
		\end{bNiceMatrix}	\begin{bNiceMatrix}
			A^{(1)}_{n+1}(0) & A^{(1)}_{n+2}(0)\\[2pt]
			A^{(2)}_{n+1}(0) &A^{(2)}_{n+2}(0)
		\end{bNiceMatrix}^{-1}\begin{bNiceMatrix}
			A^{(2)}_{n+1}\\[2pt]A^{(2)}_{n+2}
		\end{bNiceMatrix}\bigg).
\end{align*}
These three equations are the Christoffel formulas in \cite[Theorem 6]{bfm} for the  Christoffel transformation
$(\d\psi_1,\d\psi_2)\to (x\d\psi_1,x\d\psi_2)$.
\end{proof}

\end{document}